\documentclass[12pt]{amsart}

\pdfoutput=1
\usepackage{latexsym}
\usepackage[centertags]{amsmath}
\usepackage{amsfonts}
\usepackage{amssymb}
\usepackage{amsthm}
\usepackage{newlfont}
\usepackage{graphics}
\usepackage{color}
\usepackage{booktabs}
\usepackage{wrapfig}
\usepackage{subfigure}
\usepackage{graphicx}
\usepackage[usenames,dvipsnames]{xcolor}
\usepackage{graphicx}

\usepackage{wrapfig}
\usepackage[nocompress]{cite}
\makeatletter
\newcommand{\tabcaption}{\def\@captype{table}\caption}
\makeatother

\newtheorem{theo}{Theorem}
\newtheorem{defn}[theo]{Definition}
\newtheorem{exam}[theo]{Example}

\newtheorem{cor}[theo]{Corollary}
\newtheorem{prop}[theo]{Proposition}

\newtheorem{algor}[theo]{Algorithm}

\newtheorem{prob}{Problem}

\makeatletter \@addtoreset{equation}{section}
\@addtoreset{theo}{}\makeatother

\numberwithin{equation}{section}

\def\CT{\mathop{\mathrm{CT}}}
\def\N{\mathbb{N}}
\def\Z{\mathbb{Z}}

\def\N{\mathbb{N}}
\def\Z{\mathbb{Z}}

\title{A polynomial time algorithm for calculating Fourier-Dedekind sums}

\author{Guoce Xin$^{1,*}$, Xinyu Xu$^{2}$}

 \address{ $^{1,2}$School of Mathematical Sciences, Capital Normal University,
 Beijing 100048, PR China}
\email{$^1$\texttt{guoce\_xin@163.com}\ \& $^2$\texttt{xinyu0510x@163.com}}

\date{ \today}
\begin{document}
\begin{abstract}
We solve an open problem proposed in the book ``Computing the continuous discretely" written by Matthias Beck and Sinai Robins. That is,
we proposed a polynomial time algorithm for calculating Fourier-Dedekind sums. The algorithm
 is simple modular Barvinok's simplicial cone decomposition. It can be easily adapted into
 De Leora et. al.'s LattE package, which gives a nice implimentation of Barvinok's polynomial time algorithm.
\end{abstract}

\maketitle
\noindent
\begin{small}
 \emph{Mathematic subject classification}: Primary 11F20; Secondary 11Y16, 05A15, 11L03.
\end{small}

\noindent
\begin{small}
\end{small}

\section{Introduction}
This draft is an announcement. A complete version will be finished soon.

Dedekind sums are important number-theoretical objects that arise in many areas of mathematics, 
including number theory, geometry, topology, algorithmic complexity, etc. See, e.g., \cite{beck2007computing} for details and further references.
Fourier-Dedekind sums unify many variations of the Dedekind sums that have appeared in the literature, and form the building blocks of Ehrhart quasipolynomials. 
The Fourier-Dedekind sum is defined by
\begin{align}\label{FD-defn}
 s_n(a_1,a_2,\dots,a_d;b)=\frac{1}{b}\sum\limits_{k=1}^{b-1}\frac{\xi^{kn}}{(1-\xi_b^{ka_1})\cdot(1-\xi_b^{ka_2})\cdots(1-\xi_b^{ka_d})},
\end{align}
where $a_1, a_2,\dots, a_d, b\in\N$ and $b>1$ is relatively prime to each $a_i$ and $\xi_b=e^{\frac{2\pi i}{b}}$.

The following open problem about Fourier-Dedekind sum was proposed by Matthias Beck and Sinai Robins in \cite{beck2007computing}.
\begin{prob}[Open Problem]
 It is known \cite{Beck2003} that the Fourier-Dedekind sums are efficiently computable. Find a fast algorithm that can be implemented in practice.
\end{prob}

We solve this open problem by giving a desired polynomial time algorithm using a constant term concept in \cite{xin2015euclid} and a simple application of Barvinok's algorithm. The algorithm can be easily adapted into the package \texttt{LattE} by De Loera et al. \cite{LattE}.

\section{The polynomial time algorithm}
Throughout this section,, we assume that $a_1, a_2,\dots, a_d\in \N$ are coprime to $b$ unless specified otherwise.
\subsection{A brief introduction}
Here we need to write an Elliott rational function $E$ in the following form.
\begin{align}
  \label{e-E-xform}
E= \frac{L(\lambda)}{\prod_{i=1}^n (1-u_i \lambda^{a_i})}
\end{align}
where $L(\lambda)$ is a Laurent polynomial, $u_i$ are free of $\lambda$ and $a_i$ are positive integers for all $i$. The algorithm mainly relies on the following known results.
\begin{prop}\label{p-partialfraction}
Suppose the partial fraction decomposition of $E$ is given by
\begin{align}
  \label{e-E-parfrac}
E= P(\lambda)+ \frac{p(\lambda)}{\lambda^k} +\sum_{i=1}^n \frac{A_i(\lambda)}{1-u_i \lambda^{a_i}},
\end{align}
where the $u_i$'s are free of $\lambda$, $P(\lambda),p(\lambda),$ and the $ A_i(\lambda)$'s are all polynomials, $\deg p(\lambda)<k$, and
$\deg A_i(\lambda)<a_i$ for all $i$.
Then we have
$$\CT_\lambda E = P(0) + \sum_{u_i \lambda^{a_i} <1} A_i(0).$$
\end{prop}

\begin{defn}
We denote
$$\CT_\lambda \frac{1}{\underline{1-u_s \lambda^{a_s}}} E (1-u_s \lambda^{a_s}):=A_s(0).$$
In the general case, for any $\varnothing \neq I\subseteq [n]$ , we denote
$$\CT_\lambda \frac{1}{\underline{\prod\limits_{s\in I}(1-u_s \lambda^{a_s})}} E\cdot \prod\limits_{s\in I}(1-u_s \lambda^{a_s}):=\sum\limits_{s\in I}A_s(0).$$
\end{defn}

Our algorithm is based on the following observation.
\begin{prop}\label{com-F}
Suppose $F(\lambda)$ is a rational function and $F(\xi_b^k)$ exists for $k=1,2,\dots, b$, where $\xi_b=e^{\frac{2\pi i}{b}}$. Then we have
\begin{align*}
 \frac{1}{b}\sum\limits_{k=1}^{b-1}F(\xi_b^k)=\CT\limits_\lambda \frac{1}{\underline{1-\lambda^b}}F(\lambda)-\frac{1}{b}F(1).
\end{align*}
\end{prop}
\begin{proof}
The proposition follows by the following known identity.
 \begin{equation*}
\CT\limits_\lambda\frac{1}{\underline{1-\lambda^b}}F(\lambda)
=\CT\limits_\lambda\frac{1}{\underline{(1-\lambda)\cdot(1-\xi^{-1}_b\lambda)\cdots(1-\xi^{1-b}_b\lambda)}}F(\lambda)
=\frac{1}{b}\sum\limits_{k=1}^{b-1}F(\xi_b^k)+\frac{1}{b}F(1).
\end{equation*}
\end{proof}
\begin{cor}
Let $d\geq1$. Then $s_n(a_1,a_2,\dots,a_d;b)$ can be written as
\begin{align*}
  s_n(a_1,a_2,\dots,a_d;b)=\Big(Q_z-\frac{1}{b(1-z_1)\cdots(1-z_d)}\Big)\Big|_{z_i=1}.
\end{align*}
where $Q_z=\CT\limits_\lambda\frac{\lambda^n}{\underline{(1-\lambda^b)}\cdot(1-\lambda^{a_1}z_1)\cdots(1-\lambda^{a_d}z_d)}$.
\end{cor}
\begin{proof}
Apply Proposition \ref{com-F} to $F(\lambda)=\frac{\lambda^n}{(1-\lambda^{a_1}z_1)\cdots(1-\lambda^{a_d}z_d)}.$  We obtain
$$\frac{1}{b}\sum\limits_{k=1}^{b-1}\frac{\xi_b^{kn}}{(1-\xi_b^{ka_1}z_1)\cdots(1-\xi_b^{ka_d}z_d)}=\CT\limits_\lambda\frac{1}{\underline{1-\lambda^b}}
F(\lambda)-\frac{1}{b(1-z_1)\cdots(1-z_d)}.$$
Taking limits at $z_i=1$ for all $i$ gives
$$s_n(a_1,a_2,\dots,a_d;b)=\Big(\CT\limits_\lambda\frac{\lambda^n}{\underline{(1-\lambda^b)}\cdot(1-\lambda^{a_1}z_1)\cdots(1-\lambda^{a_d}z_d)}-\frac{1}{b(1-z_1)\cdots(1-z_d)}\Big)\Big|_{z_i=1},$$
as desired.
\end{proof}
\subsection{Steps of the algorithm}
We use the package \texttt{LattE} to compute $s_n(a_1,a_2,\dots,a_d;b)$.

First, by adding a slack variable $z_0$
we can write $$Q_z=\Big(\CT\limits_\lambda\frac{\lambda^n}{\underline{(1-\lambda^bz_0)}\cdot(1-\lambda^{a_1}z_1)\cdots(1-\lambda^{a_d}z_d)}\Big)\Big|_{z_0=1}.$$
For convenience, we let $$\widetilde{Q_z}:=\CT\limits_\lambda\frac{\lambda^n}{\underline{(1-\lambda^bz_0)}\cdot(1-\lambda^{a_1}z_1)\cdots(1-\lambda^{a_d}z_d)}.$$
Observe that
$$\widetilde{Q_z}=\sum\limits_{\alpha\in P\cap \Z^{d+1}}z^{\alpha}$$
enumerate lattice points in the vertex simplicial cone $P$ defined by the
 vertex $v=(-\frac{n}{b},0,\dots,0)^t$ and generators the column vectors of \begin{equation*}
   H=\left(
     \begin{array}{ccccc}
       -a_1 &-a_2    &\ldots &-a_d \\
         b  &  0     &\ldots & 0   \\
         0  &  b     &\ldots & 0    \\
     \vdots &\vdots  &\ldots &\vdots\\
       0    &  0     &\ldots & b    \\
     \end{array}
   \right).
 \end{equation*}
Then we use \texttt{LattE} to write
$$\widetilde{Q_z}=\sum\limits_{i}\widetilde{Q_i}(z_0,z_1,\cdots,z_d)$$
as a short sum of simple rational functions, and compute the limit
$$\Big(\sum\limits_{i}\widetilde{Q_i}(z_0,z_1,\cdots,z_d) -\frac{1}{b(1-z_1)\cdots(1-z_d)} \Big)\Big|_{z_j=1}.$$
This is equal to the desired $s_n(a_1,a_2,\dots,a_d;b)$.

\begin{algor}
Now we will give an algorithm for computing the Fourier-Dedekind sum $s_n(a_1,a_2,\dots,a_d;b)$.
 \begin{itemize}
   \item[1.] Add slack variable $z_0$ to $Q_z$ and get $\widetilde{Q_z}=\CT\limits_\lambda\frac{\lambda^n}{\underline{(1-\lambda^bz_0)}\cdot(1-\lambda^{a_1}z_1)\cdots(1-\lambda^{a_d}z_d)}$.
   \item[2.] We can write $\widetilde{Q_z}=\sum\limits_i\widetilde{Q_i}(z_0,\dots,z_d)$ by the \texttt{LattE} package.
   \item[3.]Eliminate slack variables $z_j$ by using either \texttt{LattE} or \texttt{CTEuclid} to give the output.
 \end{itemize}
\end{algor}

We illustrate the basic idea by using the (elementary) CTEuclid algorithm for a replacement of Step 2.3.
\begin{exam}
Compute $s_4(4,3,5;7)$.

By definition of Fourier-Dedekind sum, we have $s_4(4,3,5;7)=\frac{1}{7}\sum\limits_{k=1}^{6}\frac{\xi^{4k}}{(1-\xi_7^{4k})(1-\xi_7^{5k})(1-\xi_7^{3k})},$
where $\xi_7=e^{\frac{2\pi i}{7}}$.
\begin{align*}
Q_z&=\CT\limits_\lambda\frac{\lambda^4}{\underline{(1-\lambda^7)}\cdot(1-\lambda^4z_1)(1-\lambda^5z_2)(1-\lambda^3z_3)}\\
&={\frac {{z_{{1}}}^{9}}{ \left( {z_{{1}}}^{3}-z_{{2}} \right)  \left( z
_{{1}}z_{{3}}-1 \right)  \left( {z_{{1}}}^{7}-1 \right) }}-{\frac {z_{
{3}}}{ \left( z_{{1}}z_{{3}}-1 \right)  \left( z_{{2}}{z_{{3}}}^{3}-1
 \right)  \left( {z_{{3}}}^{7}-1 \right) }}\\
&-{\frac {{z_{{1}}}^{3}{z_{{
2}}}^{2}}{ \left( z_{{1}}z_{{3}}-1 \right)  \left( z_{{1}}{z_{{2}}}^{2
}-1 \right)  \left( {z_{{1}}}^{3}-z_{{2}} \right) }}-{\frac {{z_{{2}}}
^{2}z_{{3}}}{ \left( z_{{1}}z_{{3}}-1 \right)  \left( {z_{{2}}}^{2}-z_
{{3}} \right)  \left( z_{{2}}{z_{{3}}}^{3}-1 \right) }}\\
&+{\frac {{z_{{2
}}}^{4}}{ \left( z_{{1}}{z_{{2}}}^{2}-1 \right)  \left( {z_{{2}}}^{2}-
z_{{3}} \right)  \left( {z_{{2}}}^{7}-1 \right) }}.
\end{align*}
Then
\begin{align*}
s_4(4,3,5;7)&=\Big(Q_z-\frac{1}{(1-z_1)(1-z_2)(1-z_3)}\Big)\Big|_{z_i=1}\\
&=\Big({\frac {{z_{{1}}}^{9}}{ \left( {z_{{1}}}^{3}-z_{{2}} \right)  \left( z
_{{1}}z_{{3}}-1 \right)  \left( {z_{{1}}}^{7}-1 \right) }}-{\frac {z_{
{3}}}{ \left( z_{{1}}z_{{3}}-1 \right)  \left( z_{{2}}{z_{{3}}}^{3}-1
 \right)  \left( {z_{{3}}}^{7}-1 \right) }}\\
&-{\frac {{z_{{1}}}^{3}{z_{{
2}}}^{2}}{ \left( z_{{1}}z_{{3}}-1 \right)  \left( z_{{1}}{z_{{2}}}^{2
}-1 \right)  \left( {z_{{1}}}^{3}-z_{{2}} \right) }}-{\frac {{z_{{2}}}
^{2}z_{{3}}}{ \left( z_{{1}}z_{{3}}-1 \right)  \left( {z_{{2}}}^{2}-z_
{{3}} \right)  \left( z_{{2}}{z_{{3}}}^{3}-1 \right) }}\\
&+{\frac {{z_{{2
}}}^{4}}{ \left( z_{{1}}{z_{{2}}}^{2}-1 \right)  \left( {z_{{2}}}^{2}-
z_{{3}} \right)  \left( {z_{{2}}}^{7}-1 \right) }}-\frac{1}{(1-z_1)(1-z_2)(1-z_3)}\Big)\Big|_{z_i=1}\\
&=\frac{1}{7}.
\end{align*}

\end{exam}

\section{Appendix: Computer Experiment}
We give data using a newly developed algorithm LLLCTEuclid, which is still under construction. This part will not be included in the final version.


\mbox{\includegraphics[width=\linewidth]{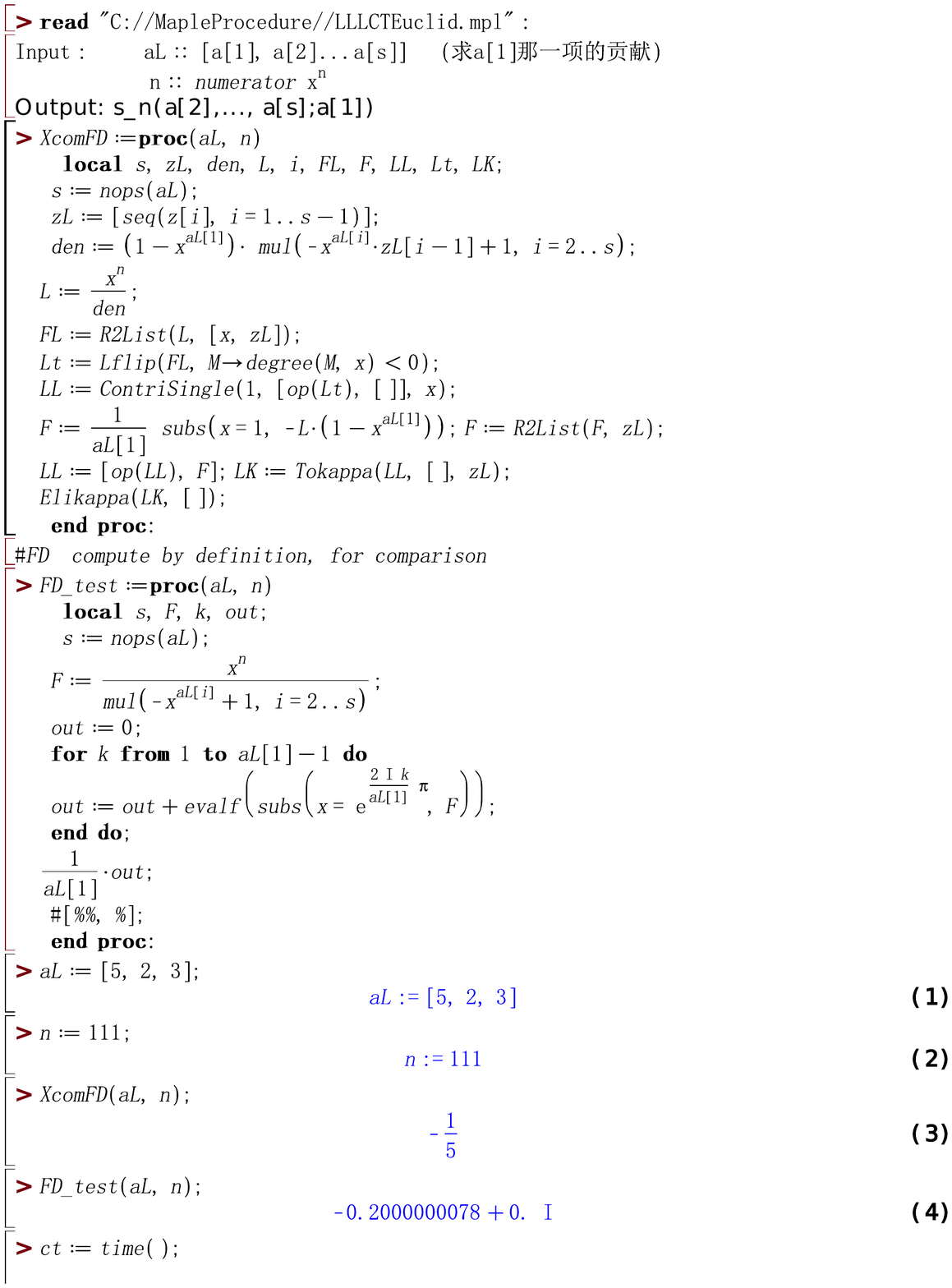}}

\mbox{\includegraphics[width=\linewidth]{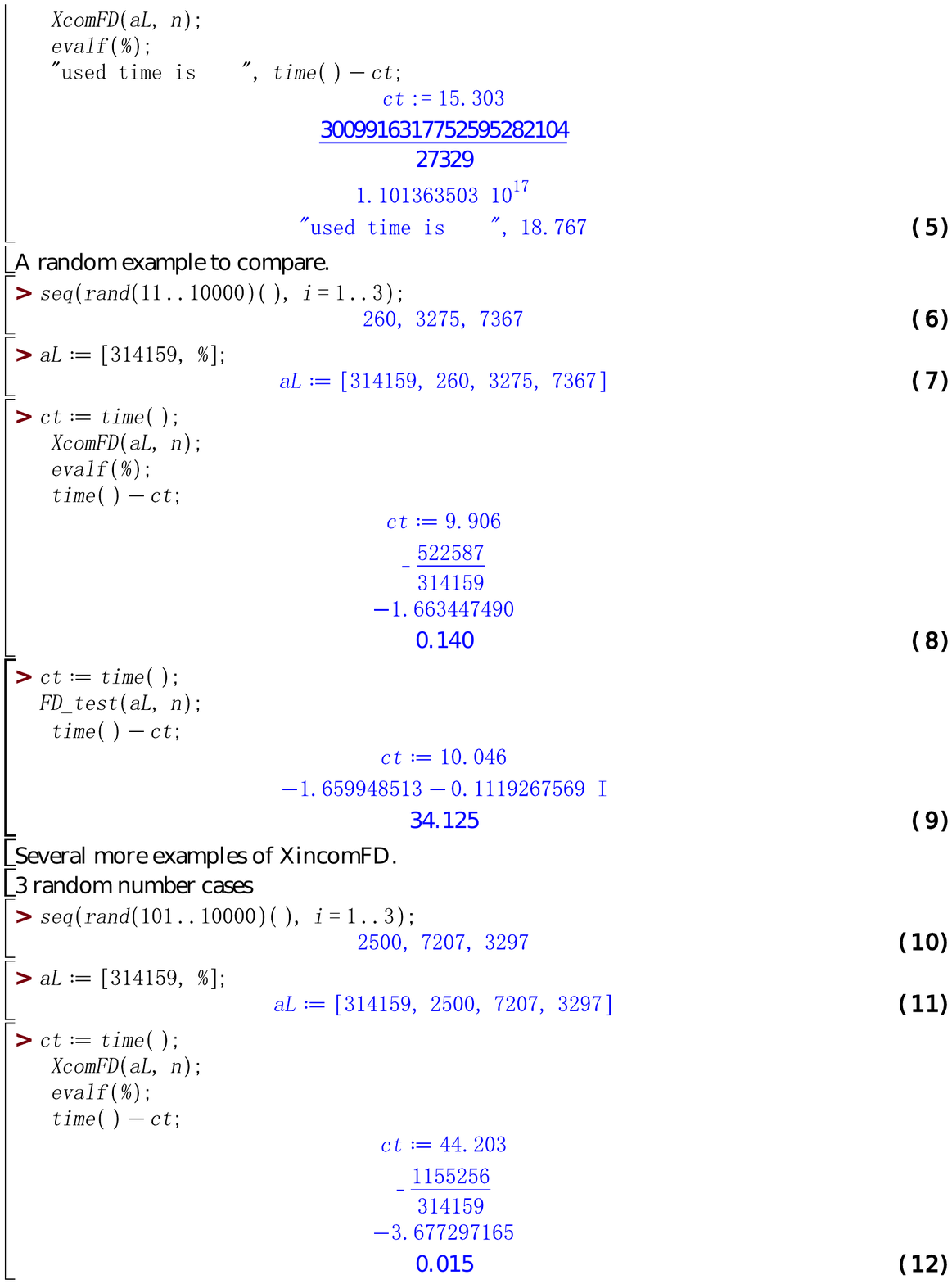}}

\mbox{\includegraphics[width=\linewidth]{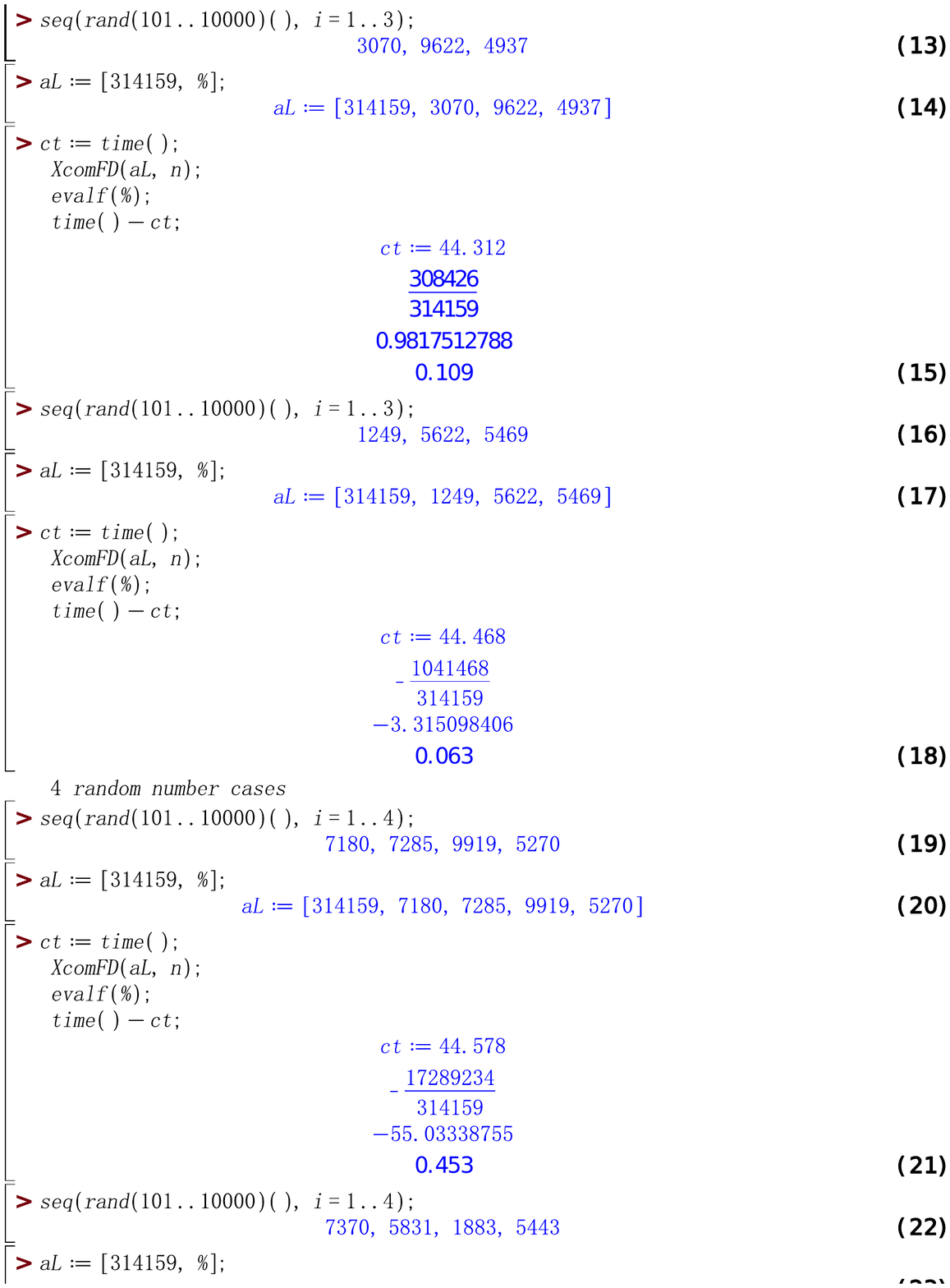}}

\mbox{\includegraphics[width=\linewidth]{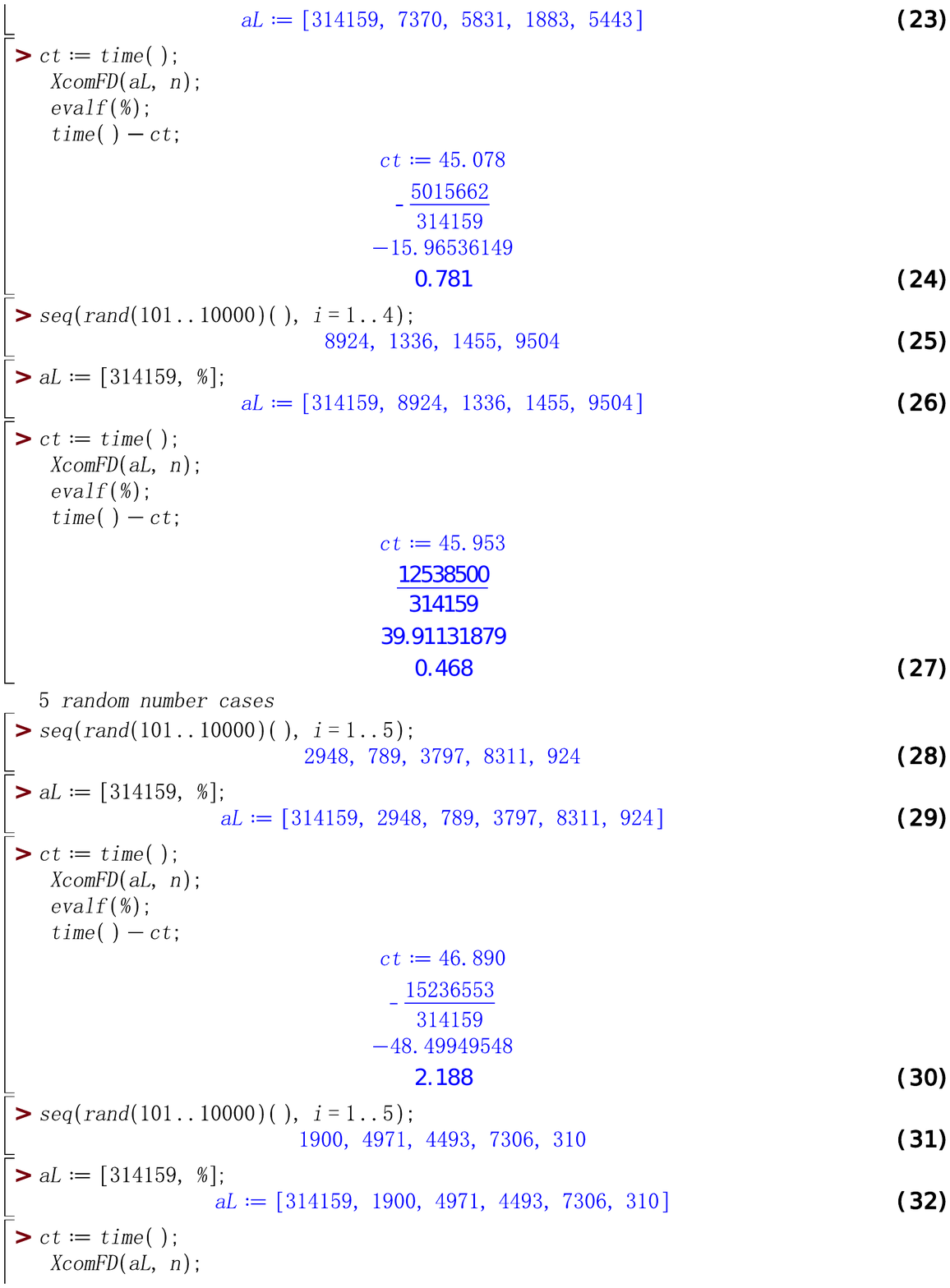}}

\mbox{\includegraphics[width=\linewidth]{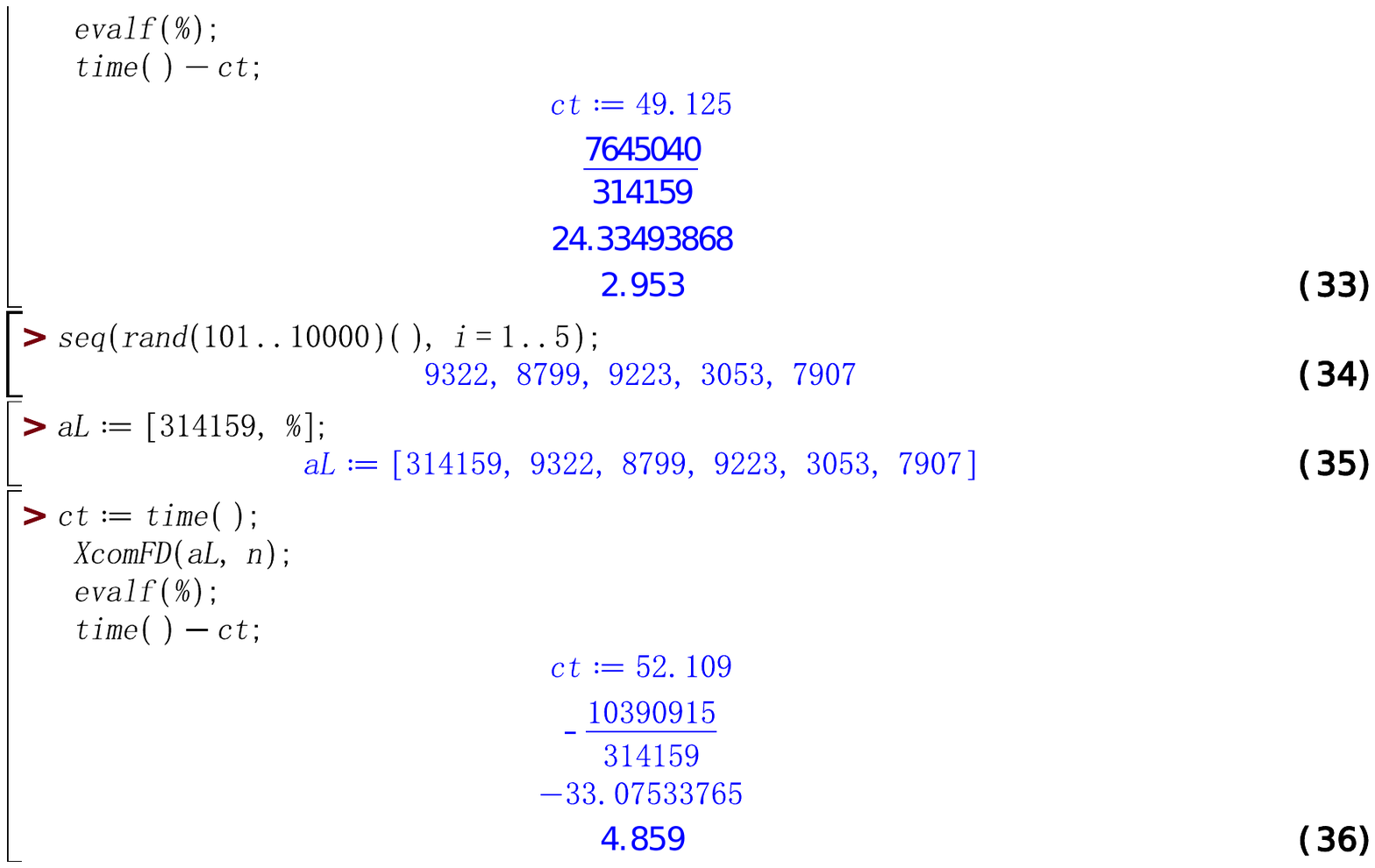}}

%
%
%
%
%
%

\end{document}